\documentclass[12pt]{amsart}

\usepackage{geometry}
\usepackage{amsmath}
\usepackage{amssymb}
\usepackage{amsthm}
\usepackage{amsfonts, dsfont}
\usepackage{mathrsfs} 
\usepackage{paralist}
\usepackage{graphics} 
\usepackage{epsfig} 
\usepackage{graphicx}  
\usepackage{epstopdf}
\usepackage{epstopdf}
\usepackage{verbatim}
\usepackage{mathrsfs}
\usepackage{mathtools}
\usepackage{pstricks}
\usepackage{relsize}
\usepackage{tikz}
\usetikzlibrary{matrix}
\usepackage{subcaption}
\usepackage{pgfplots}
\usepackage{fixltx2e}
\usepackage{enumitem}
\usepackage{bm}
\usepackage{upgreek}
\usepackage{url}
\usepackage[colorlinks=true]{hyperref}
\hypersetup{urlcolor=blue, linkcolor=blue, citecolor=red}
\usepackage{cleveref}
\usepackage{bm}
\usepackage{appendix}
\usepackage{tcolorbox} 

\usepackage{algorithm2e}
\SetKwComment{Comment}{/* }{ */}

\allowdisplaybreaks

\newcommand{\D}[1]{\mbox{\rm #1}} 
\newcommand{\dd}{\D{d}}
\newcommand{\dt}{{\rm d} t}

\newcommand{\dx}{{\rm d} x}
\newcommand{\rd}{{\rm d} }
\newcommand{\RR}{{\mathbb R}}
\newcommand{\EE}{{\mathbb E}}

\newcommand{\mm}{{\mathfrak{m}}}
\newcommand{\NN}{{\mathbb N}}

\newcommand{\OX}{{\overline{X}}}

\newcommand{\mc}[1]{\mathcal{#1}}

\newcommand{\la}{\langle}
\newcommand{\ra}{\rangle}

\DeclareMathOperator*{\argmin}{argmin}

\usepackage{amssymb}

\numberwithin{equation}{section}

\newtheorem{theorem}{Theorem}[section]
\newtheorem{lemma}[theorem]{Lemma}
\newtheorem{assum}[theorem]{Assumption}

\newtheorem{remark}[theorem]{Remark}

\definecolor{ForestGreen}{RGB}{34,139,34}
\definecolor{ao(english)}{rgb}{0.0, 0.5, 0.0}

\begin{document}

\title[Uniform-in-time mean-field limit]{Uniform-in-time mean-field limit estimate for the Consensus-Based Optimization}
\thanks{
}


\author{Hui Huang}
\author{Hicham Kouhkouh}

\address{Hui Huang \textbf{ and } Hicham Kouhkouh\newline \indent
University of Graz\newline \indent
Department of Mathematics and Scientific Computing, NAWI Graz\newline \indent
Graz, Austria
}
\email{\texttt{hui.huang@uni-graz.at}, \quad \texttt{hicham.kouhkouh@uni-graz.at}}
\thanks{}







\date{\today}

\begin{abstract}
We establish a uniform-in-time estimate for the mean-field convergence of the Consensus-Based Optimization (CBO) algorithm by rescaling the consensus point in the dynamics with a small parameter $\kappa \in (0,1)$. This uniform-in-time estimate is essential, as CBO convergence relies on a sufficiently large time horizon and is crucial for ensuring stable, reliable long-term convergence, the latter being key to the practical effectiveness of CBO methods.
\end{abstract}










\maketitle

\section{Introduction}

In recent years, optimization methods have become central to numerous fields, from machine learning and data science to engineering and economics. Among these, Consensus-Based Optimization (CBO) \cite{pinnau2017consensus,carrillo2018analytical} has emerged as a promising approach for addressing challenging nonconvex and nonsmooth optimization problems. Inspired by the collective behavior observed in social and biological systems, CBO utilizes the interactions of multiple agents, or particles,  to search for optimal solutions in a coordinated and effective manner.  Driven by a wide range of applications, researchers have extended and adapted the original CBO model to address various settings. These extensions include global optimization on compact manifolds \cite{fornasier2021JMLR, ha2022stochastic}, handling general constraints \cite{borghi2023constrained, beddrich2024constrained}, cost functions with multiple minimizers \cite{bungert2022polarized}, multi-objective problems \cite{borghi2022consensus}, and sampling from distributions \cite{carrillo2022consensus}. CBO methods have also been applied in multi-player games \cite{chenchene2023consensus}, min-max problems \cite{huang2024consensus,borghi2024particle}, multi-level optimization  \cite{herty2024multiscale}, and clustered federated learning \cite{carrillo2023fedcbo}.

Let us consider  the following optimization problem:
\begin{equation*}
	\text{Find }\; x^* \in \argmin_{x\in \RR^d} f(x),\,
\end{equation*}
where $f(\cdot)$ can be a non-convex non-smooth objective function that one wishes to minimize. Then the CBO dynamic considers the following system of $N$ interacting particles, denoted as $\{X_{\cdot}^i\}_{i=1}^N$, which satisfies
\begin{equation}\label{CBOparticle}
	\rd X_t^i=-\lambda(X_t^i - \mathfrak{m}_\alpha(\rho_t^{N}))\dt+\sigma D(X_t^i-\mathfrak{m}_\alpha(\rho_t^{N}))\rd B_t^{i},\quad i=1,\dots, N=:[N]\,,
\end{equation}
where  $\sigma>0$ is a real constant, $\rho_t^N:=\frac{1}{N}\sum_{i=1}^N\delta_{X_t^i}$ is the empirical measure associated to the particle system, and $\{B_.^i\}_{i=1}^N$ are $N$ independent $d$-dimensional Brownian motions.
 Here, we employ the anisotropic diffusion in the sense that $D(X):=\mbox{diag}(|X_1|,\dots,|X_d|)$ for any $X\in\RR^d$, which has been proven to handle high-dimensional problems more effectively \cite{carrillo2021consensus,fornasier2022anisotropic}.
The \textit{current global consensus} point $\mathfrak{m}_{\alpha}(\rho_t^N)$ is defined by
\begin{equation}\label{XaN}
	\mathfrak{m}_{\alpha}(\rho_t^N): = \frac{\int_{\RR^d} x \, \omega_{\alpha}^{f}(x)\; \rho_t^N(\dx)}{\int_{\RR^d}\omega_{\alpha}^{f}(x)\; \rho_t^N(\dx)}\,,
\end{equation}
and the weight function is chosen to be $\omega_\alpha^f(x):=\exp(-\alpha f(x))$.
This choice of weight function is motivated by the well-known Laplace's principle \cite{miller2006applied,MR2571413}. Moreover, the particle system is initialized with i.i.d. data $\{X_{0}^i\}_{i=1}^N$, where each $X_0^i$ is distributed according to a given measure $\rho_0 \in \mathscr{P}(\RR^d)$ for all $i \in [N]$.

The convergence proof of the CBO method is typically conducted in the context of mean-field theory \cite{fornasier2024consensus,carrillo2018analytical}. Specifically, rather than analyzing the $N$-particle system \eqref{CBOparticle} directly, one considers the limit as $N$ approaches infinity and examines the corresponding McKean-Vlasov process $\OX_{\cdot}$, which is governed by the following equation:
\begin{equation}\label{CBO}
\begin{aligned}
    \rd \OX_t & =-\lambda(\OX_t-\mathfrak{m}_{\alpha}(\rho_t)) \, \dt + \sigma D(\OX_t-\mathfrak{m}_{\alpha}(\rho_t)) \, \rd B_t\,, \quad \text{with } \; \rho_t:= \text{Law}(\OX_{t})
\end{aligned}
\end{equation}
Then, under certain assumptions on the objective function $f(\cdot)$, and the well-prepared initial data and parameters, it can be proved  that, for any fixed $\alpha>0$, $\rho_t$ converges to a Dirac measure. If moreover $\alpha$ is chosen large enough, it can also be shown that the latter Dirac measure can be supported on a point close to $x^*$, a global minimizer of $f(\cdot)$. One may also obtain the global convergence at the particle level \eqref{CBOparticle} directly, as shown in \cite{ha2020convergence,ha2021convergence}. However, this approach requires common noise in the particle system; that is, it assumes \( B^i_. = B_. \) for all \( i \in [N] \) in \eqref{CBOparticle}.

\subsection{Related results}

The convergence of the particle system \eqref{CBOparticle} to the McKean-Vlasov process \eqref{CBO} is referred to as the mean-field limit. Establishing a rigorous proof of this convergence is challenging, as the consensus point defined in \eqref{XaN} is only locally Lipschitz. Several mean-field results for CBO have been established \cite{fornasier2020consensus,huang2022mean, fornasier2024consensus, gerber2023mean}. In \cite{fornasier2020consensus}, a mean-field limit estimate for a variant of CBO constrained to hypersurfaces is achieved through the coupling method. Later, \cite{huang2022mean} proves the mean-field limit for standard CBO using a compactness argument based on Prokhorov's theorem; however, this approach does not provide a convergence rate in terms of the number of particles $N$. The results in \cite{fornasier2024consensus}  demonstrate a probabilistic mean-field approximation of the form
\begin{equation}\label{meaneq1}
\begin{aligned}
	& \quad \quad \sup_{i\in[N]} \, \sup_{t\in[0,T]}\EE\left[|X_t^i-\OX_t^i|^2 \,\textbf{I}_{\Omega/\Omega_M}\right]\leq C\frac{1}{N}, \quad \text{where }\\
     &  \Omega_M:=\left\{\sup_{t\in[0,T]}\frac{1}{N}\sum_{i\in[N]}(|X_t^i|^4+|\OX_t^i|^4)\geq M\right\} \quad \text{ and } \quad \mathbb{P}(\Omega_M)=\mc{O}(M^{-1}).
\end{aligned}
\end{equation} 
Most recently, \cite{gerber2023mean} removed the high-probability assumption by establishing an improved stability estimate on $\mathfrak{m}_{\alpha}$ (see \eqref{lemeq1}), resulting in
\begin{equation}\label{meaneq2}
\sup_{i\in[N]}\EE\left[\sup_{t\in[0,T]}|X_t^i-\OX_t^i|^2\right]\leq C\frac{1}{N}\,.
\end{equation}
This result is subsequently extended to the multi-species case in \cite{huang2025well}. However, none of the existing results in the literature provide a uniform-in-time estimate for the mean-field limit of the CBO; specifically, the constant $C$ in \eqref{meaneq1} and \eqref{meaneq2} depends on $T$ exponentially. Such a uniform-in-time estimate is crucial, as the convergence of CBO fundamentally relies on $t$ being sufficiently large. It is also highly sought after within the CBO community because it directly addresses the stability and reliability of long-time convergence, a key factor in ensuring the practical effectiveness of CBO methods.

After the completion of our present work, two new interesting results appeared in \cite{bayraktar2025uniform} and then in \cite{gerber2025uniform}: \\
\textbullet\; The authors in \cite{bayraktar2025uniform} 
succeeded to prove a uniform--in--time mean--field limit for a CBO model where the diffusion coefficient is truncated. This modification allows the dynamics to remain confined within a compact set, and enhances its stability properties. Their main result \cite[Theorem 2.6]{bayraktar2025uniform} is a \textit{weak propagation of chaos} of the form 
\begin{equation*}
    \sup\limits_{t\geq 0} \left| \mathbb{E}\left[\mathbf{\Phi}(\nu^{N}_t)\right] \,-\, \mathbf{\Phi}(\bar{\nu}_{t}) \right| \leq \frac{C}{N}
\end{equation*}
where $\nu^{N}_{t}$ is the empirical measure of the interacting particles, and $\bar{\nu}_{t}$ is the mean--field law. Here  $C>0$ is a constant independent of time, and $\mathbf{\Phi}$ is any smooth enough function chosen within a given class of functions from $\mathscr{P}(\mathbb{R}^{d})$ to $\mathbb{R}$. The techniques used in \cite{bayraktar2025uniform} are different from ours, and rely on new results about linearization of nonlinear Fokker--Planck equations. \\
\textbullet\; On the other hand, in  \cite[Theorem 2.1]{gerber2025uniform}, the authors were able to prove a uniform--in--time mean--field limit in the sense of \eqref{meaneq2}. Their analysis takes advantage from new stability estimates for the consensus term, together with concentration estimates for the interacting particle system. These were obtained in particular using synchronous coupling. An advantage of their main result, is that it concerns the original model ($\kappa=1$ in \eqref{CBO kappa}) of CBO without further modifications. 

Our strategy and main result are discussed next.

\subsection{A key observation}

We believe that the primary difficulty in obtaining a uniform--in--time estimate stems from the non-uniqueness of the invariant measure of \eqref{CBO}. Indeed, for the process \eqref{CBO}, it can be easily verified that  any Dirac measure (not necessarily supported on the global minimizer) is invariant; see Remark \ref{rem: dirac inv} below. This non--uniqueness is mainly due to two reasons: 
in the drift, the term $-X$ is not strong enough with respect to the mean--field term $\mathfrak{m}_{\alpha}(\cdot)$, and the diffusion degenerates. With these features, it becomes difficult to study the dynamical properties of \eqref{CBO}, in particular its long-time behavior.

To remedy to this situation, we consider a modification of \eqref{CBO} which addresses exactly these two mentioned issues, and which has already been proposed in \cite{huang2024self, herty2024multiscale}. Hence, we consider a \textit{rescaled} CBO given by
\begin{equation}\label{CBO kappa}
    \rd \OX_t =-\lambda(\OX_t - \kappa\,\mathfrak{m}_{\alpha}(\rho_t))\dt+\sigma\left(\delta\,\mathds{I}_{d} + D(\OX_t - \kappa\, \mathfrak{m}_{\alpha}(\rho_t)) \right)\rd B_t
\end{equation}
complemented with an initial condition $\OX_{0}\sim \rho_0\in\mathscr{P}_{16}(\RR^d)$, and where $\mathds{I}_{d}$ is the $d$-dimensional identity matrix, $0<\kappa<1$ is a small positive constant, and $\delta\geq 0$ is a constant that is allowed to be null\footnote{We keep $\delta$ in the present manuscript to capture the models in \cite{huang2024self, herty2024multiscale} where it is assumed to be positive. Yet, $\delta$ does not need to be non-null for our results to hold in the sequel. In either cases, $\delta$ is a constant  independent of the data of the problem (in particular of $T,N$), and is kept fixed throughout the manuscript.}. 
Clearly, the  underlying system of $N$ interacting particles is of the form
 \begin{equation}\label{CBOkappa particle}
\rd X_t^i =-\lambda \left(X_t^i - \kappa\,\mathfrak{m}_{\alpha}(\rho_t^N) \right) \dt+\sigma\left(\delta\,\mathds{I}_{d} + D(X_t^i - \kappa\, \mathfrak{m}_{\alpha}(\rho_t^N) \right)\rd B_t^i,\quad i\in [N]\,.
\end{equation}
The advantage of the  rescaled CBO \eqref{CBO kappa} is that now one can prove that it has a unique invariant measure $\rho_\alpha^*$ \cite[Proposition 3.4]{huang2024self}. With the invariant measure $\rho_\alpha^*$ of \eqref{CBO kappa} in hand, we can write a formal global convergence result. Indeed, let 
$\rho_\alpha^*=\text{Law}[\OX_\infty]$, then the following holds
\begin{equation}
	\label{eq:formal conv}
	0=\frac{\rd \EE[\OX_\infty]}{\dt}=-\lambda (\EE[\OX_\infty]-\kappa\,\mathfrak{m}_{\alpha}(\rho_\alpha^*)) \; \Rightarrow \;  \EE[\OX_\infty] 
	= \kappa \,\mathfrak{m}_{\alpha}(\rho_\alpha^*)=\kappa\int_{ \RR^d }x\,\eta_\alpha^*(\dx)\,,
\end{equation}
where $\eta_\alpha^*(\dx):=\frac{\omega_\alpha^f(x)\rho_\alpha^*(\dx)}{\int_{ \RR^d }\omega_\alpha^f(x)\rho_\alpha^*(\dx)}$.
If additionally we assume that for any $\varepsilon>0$, there exists some constant $C_\varepsilon>0$ independent of $\alpha$ such that 
\begin{equation*}
    \rho_\alpha^*(\mc A_\varepsilon)\geq C_\varepsilon\, \quad \text{ with } \quad \mc A_\varepsilon:=\left\{x\in\RR^d:~e^{-f(x)}>e^{-f(x^*)}-\varepsilon \right\}
\end{equation*}
where $x^{*}$ is a global minimizer, then according to  \cite[Lemma A.3]{huang2024consensus} it holds that
\begin{equation}\label{lap_princ}
	\lim\limits_{\alpha\to\infty}\left(-\frac{1}{\alpha}\log\left(\int_{ \RR^d }\omega_\alpha^f(x)\rho_\alpha^*(\dx)\right)\right)= f(x^*)\,.
\end{equation}
In order to better see how  $\eta_\alpha^*(\dx):=\frac{\omega_\alpha^f(x)\rho_\alpha^*(\dx)}{\int_{ \RR^d }\omega_\alpha^f(x)\rho_\alpha^*(\dx)}$ approximates the Dirac distribution $\delta_{x^*}$ for large $\alpha\gg 1$, we proceed with the following computations. Recall Laplace's principle in \eqref{lap_princ}, which is equivalent to 
\begin{equation*}
    \lim\limits_{\alpha\to\infty}\left(\int_{ \RR^d }\omega_\alpha^f(x)\rho_\alpha^*(\dx)\right)^{\frac{1}{\alpha}}= e^{-f(x^*)}\,,
\end{equation*}
which also means
\begin{equation*}
\begin{aligned}
    \lim\limits_{\alpha\to\infty} \frac{e^{-f(x^*)}}{\left(\int_{ \RR^d }\omega_\alpha^f(x)\rho_\alpha^*(\dx)\right)^{\frac{1}{\alpha}}} = 1,
\end{aligned}
\end{equation*}
or equivalently (by rising to power $\alpha$)
\begin{equation*}
\begin{aligned}
    \lim\limits_{\alpha\to\infty} \frac{e^{-\alpha f(x^*)}}{\int_{ \RR^d }\omega_\alpha^f(x)\rho_\alpha^*(\dx)} = 1.
\end{aligned}
\end{equation*}
It suffices now to note that the left-hand side is the integration of the indicator function $\textbf{I}_{\{x^*\}}(\cdot)$ supported on the global minimizer $x^{*}$ with respect to $\eta_\alpha^*(\dx)$. Indeed, 
we have
\begin{equation*}
    \frac{e^{-\alpha f(x^*)}}{\int_{ \RR^d }\omega_\alpha^f(x)\rho_\alpha^*(\dx)}  = \la \eta_\alpha^*(\dx),\textbf{I}_{\{x^*\}}\ra
\end{equation*}
where $\langle\cdot\,,\,\cdot\rangle$ denotes the integration of a function with respect to a measure. Therefore, when $\alpha$ is large enough, one expects $\la \eta_\alpha^*(\dx),\textbf{I}_{\{x^*\}}\ra \approx 1$. 
Thus $\eta_\alpha^*$ approximates the Dirac distribution $\delta_{x^*}$ for large $\alpha\gg 1$.  Consequently, $\mathfrak{m}_{\alpha}(\rho_\alpha^*)$ provides a good estimate of $x^*$, which (recalling \eqref{eq:formal conv}) leads to the fact that $\EE[\OX_\infty]\approx \kappa\, x^*$ for sufficiently large $\alpha\gg 1$. The rigorous proof is obtained in \cite{huang2025faithful}.

\begin{remark}\label{rem: dirac inv}
    We can verify our claim of \eqref{CBO} having all Dirac measures invariant. 
    To do so, let us introduce the following matrix-valued and vector-valued functions defined for $(x,\mu) \in \mathbb{R}^d \times \mathscr{P}(\mathbb{R}^d)$ by
    \begin{align*}
    	&A:=(a^{ij})_{i,j}, \quad \text{ such that } \quad a^{ij}(x,\mu):=\frac{\sigma^2}{2}(|(x-\mathfrak{m}_{\alpha}(\mu))_i|)^2\delta_{ij} \\
        \mbox{ and } \quad & \; b :=(b^i)_{i},  \quad \text{ such that } \quad b^i(x,\mu):=- \lambda(x-\mathfrak{m}_{\alpha}(\mu))_i
    \end{align*}
    $\delta_{ij}$ being the Kronecker symbol. 
    The diffusion operator whose coefficients are $A$ and $b$ is given by
    \begin{equation*}
        L_{A,b,\mu}\psi(x) =  \text{trace}\big(A(x,\mu)D^2\psi(x)\big) + 
        b(x,\mu) \cdot \nabla\psi(x) \quad \forall \; \psi \in \mathscr{C}_c^\infty (\mathbb{R}^d),
    \end{equation*}
    and $L_{A,b,\mu}^{*}$ is its adjoint operator. 
    
    A measure $\mu$ is invariant if it solves $\; L_{A,b,\mu}^{*}\mu=0 \; $ in the distributional sense that is
    \begin{equation}\label{dist}
        \int_{\mathbb{R}^d} L_{A,b,\mu}\psi(x)\; \dd \mu(x) =0 \quad \forall \; \psi \in \mathscr{C}_c^\infty (\mathbb{R}^d).
    \end{equation}
    
    We can now verify \eqref{dist} when evaluated in a Dirac measure supported in an arbitrarily chosen point, say $\bm{z}\in \mathbb{R}^{d}$. We have $\bm{\delta_{z}}(x) = 0$ whenever $x\neq \bm{z}$ and $\bm{\delta_{z}}(\bm{z}) = 1$. In particular, we have
    \begin{align*}
        \mathfrak{m}_{\alpha}(\bm{\delta_{z}}) = \frac{
        \int_{\mathbb{R}^d} x \, \omega_{\alpha}^{f}(x)\; \bm{\delta_{z}}(\dd x)}{\int_{\mathbb{R}^d}\omega_{\alpha}^{f}(x)\; \bm{\delta_{z}}(\dd x)} = \bm{z}.
    \end{align*}
    The coefficients when evaluated in $\mu = \bm{\delta_{z}}$ become
    \begin{equation}\label{evaluation}
    \begin{aligned}
    	& a^{ij}(x,\bm{\delta_{z}}):=\frac{\sigma^2}{2}|(x-\mathfrak{m}_{\alpha}(\bm{\delta_{z}}))_i|^2\delta_{ij}  \; = \; \frac{\sigma^2}{2}|(x- \bm{z}))_i|^2\delta_{ij}  \\
        \mbox{ and }\quad  & b^i(x,\bm{\delta_{z}}):=- \lambda(x-\mathfrak{m}_{\alpha}(\bm{\delta_{z}}))_i \; = \; - \lambda(x-\bm{z})_i
    \end{aligned}
    \end{equation}
    Finally, $\forall \; \psi \in \mathscr{C}_c^\infty (\mathbb{R}^d)$, we have
    \begin{align*}
        \int_{\mathbb{R}^d} L_{A,b,\bm{\delta_{z}}}\psi(x)\; \bm{\delta_{z}}(\dd x)  & =  \int_{\mathbb{R}^d} \bigg[\text{trace}\big(A(x,\bm{\delta_{z}})D^2\psi(x)\big) + 
        b(x,\bm{\delta_{z}}) \cdot \nabla\psi(x)\bigg]\; \bm{\delta_{z}}(\dd x) \\
        & =   \text{trace}\big(A(\bm{z},\bm{\delta_{z}})D^2\psi(\bm{z})\big) + 
        b(\bm{z},\bm{\delta_{z}}) \cdot \nabla\psi(\bm{z}). 
    \end{align*}
    But, using \eqref{evaluation}, it holds $A(\bm{z},\bm{\delta_{z}}) = 0$ and $b(\bm{z},\bm{\delta_{z}}) = 0$. Thus we have 
    \begin{equation*}
        \int_{\mathbb{R}^d} L_{A,b,\bm{\delta_{z}}}\psi(x)\; \bm{\delta_{z}}(\dd x) = 0 \quad \forall \; \psi \in \mathscr{C}_c^\infty (\mathbb{R}^d), \; \forall \; \bm{z} \in \mathbb{R}^{d}
    \end{equation*}
    which proves the claim. Note that with $\kappa <1$, the same computations (in particular \eqref{evaluation}) show that Dirac measures can never be invariant for such dynamics, except for a Dirac measure supported in $0$. If moreover we have $\delta >0$ in \eqref{CBO kappa}, then also Dirac measure supported in $0$ is no more invariant. 
\end{remark}

Numerically, in the Figure \ref{fig:ackley}, we apply our rescaled CBO particle system to the Ackley function 
\[
A(x) := -20 \exp\left(-0.2 |x-3|\right) - \exp\left( \cos(2 \pi (x-3))\right) + e + 20\,,
\]
which is a popular benchmark nonconvex function in optimization with multiple local minimizers and a unique global minimizer $x^*=3$.
\begin{figure}[h]
	\centering
	\includegraphics[width=0.65\textwidth]{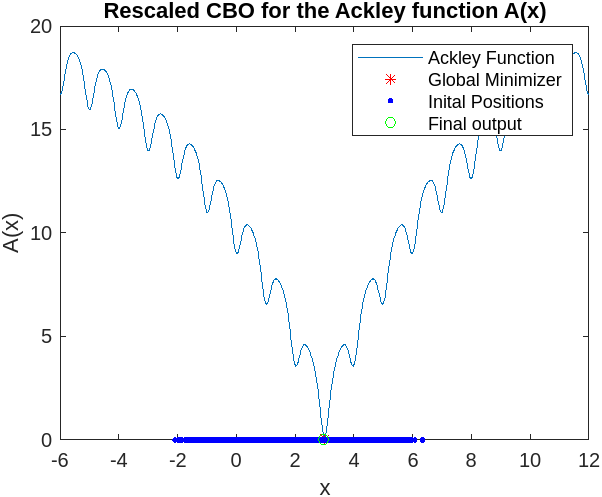}
	\caption{We apply the rescaled CBO particle system \eqref{CBOkappa particle} to the Ackley function $A(x)$, which has a unique global minimizer $x^*=3$ (the red star). The initial particles (the blue dots) are sampled from a normal distribution $\mathcal{N}(2,1)$. The simulation parameters are $N=10^5,\lambda=1,\sigma=2,\alpha=10^{15},\delta=0, \rd t=0.01, T=100$ and $\kappa=0.01$. The final output is $\frac{1}{\kappa}\frac{1}{N}\sum_{i\in[N]}X_T^{i}\approx \frac{1}{\kappa}\EE[\OX_T]$ (the green circle).}
	\label{fig:ackley}
\end{figure}

\subsection{Main contribution} We shall prove a uniform-in-time estimate for the mean-field convergence of the rescaled CBO particle system \eqref{CBOkappa particle} (Theorem \ref{thmmean}). This is based on a uniform-in-time bound on the moments of the dynamics (Theorem \ref{thm: finite moments}), together with recent results on the probability of large excursions \cite{gerber2023mean}, and an earlier result on the mean-field convergence \cite{doukhan2009evaluation}. A key observation is that rescaling the consensus point strengthens the confining properties of the dynamics. Finally, we conclude with Algorithm \ref{alg: cbo k}, highlighting how the proposed modification of the dynamics can still fit within the existing algorithms of CBO.

\section{Uniform-in-time mean-field limit estimate}

\begin{assum}\label{assum1}
	We assume the following properties for the objective function.
	\begin{enumerate}
		\item $f:~\RR^d\to \RR$ is bounded from below by $\underline f=\min f$, and there exist $L_{f}>0, s\geq 0$ constants such that
		\begin{equation*}
			|f(x)-f(y)|\leq L_f(1+|x|+|y|)^s|x-y|, \quad \forall x,y\in \RR^d\,.
		\end{equation*} 
		\item There exist constants $c_\ell,C_\ell, c_u, C_u>0$  and $\ell>0$ such that
        \begin{equation*}
        \begin{aligned}
			f(x) -\underline f & \leq c_u|x|^\ell+C_u,\quad \forall\, x\in \RR^d,\\
            f(x) -\underline f	& \geq c_\ell |x|^\ell-C_\ell,\quad \forall\, x\in \RR^d.
        \end{aligned}
		\end{equation*}
	\end{enumerate}
\end{assum}

In what follows, $\|\cdot\|$ denotes the Frobenius norm of a matrix and $|\cdot|$ is the standard Euclidean norm in $\RR^d$; $\mathscr{P}(\RR^d)$ denotes the space of probability measures on $\RR^d$, and $\mathscr{P}_p(\RR^d)$ with $p\geq 1$ contains all $\mu\in \mathscr{P}(\RR^d)$ such that $\mu(|\cdot|^p):=\int_{\RR^d}|x|^p\mu(\dx)<\infty$; it is equipped with $p$-Wasserstein distance $W_p(\cdot,\cdot)$. 
Lastly, we define $\mathscr{P}_{2,R}(\mathbb{R}^{d}) := \{\mu \in \mathscr{P}_{2}(\mathbb{R}^{d})\,:\, \mu(|\cdot|^{2})\leq R\}$.

First, let us recall some estimates on $\mathfrak{m}_{\alpha}(\mu)$ from \cite[Corollary 3.3, Proposition A.3]{gerber2023mean}.
\begin{lemma}\label{lem: useful estimates}
Suppose that $f(\cdot)$ satisfies Assumption \ref{assum1}. Then for all $R>0$ and for all $p\geq 1$, 
there exists some constant $L_{\mathfrak{m}}>0$ depending only on $L_f,s,R,p,\alpha$ such that
\begin{equation}\label{lemeq1}
    |\mathfrak{m}_{\alpha}(\mu)-\mathfrak{m}_{\alpha}(\nu)|\leq L_{\mathfrak{m}} W_p(\mu,\nu)\quad \forall (\mu,\nu)\in \mathscr{P}_{p,R}(\RR^d)\times \mathscr{P}_{p}(\RR^d)\,.
\end{equation}
Moreover for all $q\geq p\geq 1$, there exists constant $C_1>3$ depending only on $p,q,c_\ell,c_u$, $C_\ell,C_u,\ell,\alpha$ such  that
\begin{equation}\label{lemeq2}
    |\mathfrak{m}_{\alpha}(\nu)|^p\leq C_1\left(\int_{\RR^d}|x|^q\nu(dx)\right)^{\frac{p}{q}}\quad \forall \nu\in \mathscr{P}_q(\RR^d)\,.
\end{equation}
\end{lemma}

\begin{lemma}\label{thmexist}  
	Let $f(\cdot)$ satisfy Assumption \ref{assum1} and $\rho_0\in\mathscr{P}_{p}(\RR^d)$ for any $p\geq 2$, then both the dynamics \eqref{CBOkappa particle} and \eqref{CBO kappa} have unique strong solutions and they satisfy that
	\begin{equation}\label{mbound}
		\sup_{i\in[N]}\EE\left[\sup_{t\in[0,T]}|X_t^i|^{p}\right],\quad \EE\left[\sup_{t\in[0,T]} |\OX_t|^{p}\right]<\infty
	\end{equation}
	for any $0<T<\infty$. 
\end{lemma} 
\begin{proof}
The proof of well-posedness for the particle system \eqref{CBOkappa particle} follows the same reasoning as in \cite[Theorem 2.1]{carrillo2018analytical} or \cite[Theorem 2.2]{gerber2023mean}, relying on the fact that the coefficients in \eqref{CBOkappa particle} are locally Lipschitz. For the mean-field dynamic \eqref{CBO kappa}, well-posedness can be established using the Leray-Schauder fixed-point theorem, as demonstrated in \cite[Theorem 3.1]{carrillo2018analytical} or \cite[Theorem 2.2]{gerber2023mean}. Finally, the moment estimates in \eqref{mbound} are derived from \cite[Lemma 3.4]{carrillo2018analytical} or \cite[Lemma 3.5]{gerber2023mean} by using the Burkholder-Davis-Gundy inequality.
\end{proof}

Now let $\{(\OX_t^i)_{t\geq 0}\}_{i=1}^N$ be $N$ independent copies of solutions to the mean-field dynamics \eqref{CBO kappa}, so they are i.i.d. with the same  distribution $\rho_t$ at any time $t\geq 0$. First, one can prove the following uniform-in-time moment estimates for empirical measures.

\begin{theorem}\label{thm: finite moments}
	Let $f(\cdot)$ satisfy assumption \ref{assum1} and it satisfies $\rho_0\in\mathscr{P}_{2p}(\RR^d)$ with any $p\geq 2$. Suppose that $\lambda> C(p,d,\sigma)$ for some $C(p,d,\sigma)>0$ depending only on $p,d,\sigma$.
	\begin{enumerate}
		\item Consider particle system \eqref{CBOkappa particle} with $\rho_0^{\otimes N}$-distributed initial data, and let $\rho_t^N$ be the corresponding empirical measure at time $t\geq 0$. Then for sufficiently small $\kappa$ there exists some constant $C_2>0$ depending only $p,\lambda,\sigma,C_1,d,\alpha$ and $\rho_0(|\cdot|^p)$ such that it holds
		\begin{equation}\label{eq(1)}
			\sup_{t\geq 0}\left\{\sup_{i\in[N]}\EE[|X_t^i|^p]+\EE[\rho_t^N(|\cdot|^p)]+\EE[|\mm_\alpha(\rho_t^N)|^p]\right\}<C_2\,.
		\end{equation}
		\item Consider particles  $\{(\OX_t^i)_{t\geq 0}\}_{i=1}^N$  with $\rho_0^{\otimes N}$-distributed initial data satisfying \eqref{CBO kappa}, and let $\bar \rho_t^N$ be the corresponding empirical measure. Then  for sufficiently small $\kappa$ there exists some constant $C_2>0$ depending only $p,\lambda,\sigma,C_1,d,\alpha$ and $\rho_0(|\cdot|^p)$  such that it holds 
		\begin{equation}\label{eq(2)}
			\sup_{t\geq 0}\left\{\sup_{i\in[N]}\EE[|\OX_t^i|^p]+\EE[\bar\rho_t^N(|\cdot|^p)]+\EE[|\mm_\alpha(\bar \rho_t^N)|^p]\right\}<C_2\,.
		\end{equation}
	\end{enumerate}
\end{theorem}

\begin{proof}
Arguments for $(1)$ and $(2)$ are parallel, so we only prove $(1)$. We shall use the notation $U\otimes V = (u_{m}v_{n})_{1\leq m,n\leq d}$ for any $U,V \in \mathbb{R}^{d}$.  Let $h(x):=\frac{1}{p}|x|^{p}$, for $x\in \mathbb{R}^{d}$. Then for any fixed $i\in[N]$, applying It\^{o}'s formula leads to
\begin{align*}
    \rd h(X^{i}_{t}) & = \bigg\{ -\lambda \nabla h(X^{i}_{t})\cdot \big(X^{i}_{t} - \kappa \, \mathfrak{m}_{\alpha}(\rho^{N}_{t})\big) \\
    & \quad \quad \quad + \frac{\sigma^{2}}{2} \text{trace}\left[D^{2}h(X^{i}_{t}) \bigg( \delta\,\mathds{I}_{d} + D\big(X^{i}_{t} - \kappa\,\mathfrak{m}_{\alpha}(\rho^{N}_{t})\big) \bigg)^{2} \right] \bigg\}\,\rd t \\
    & \quad \quad \quad \quad \quad \quad + \sigma\,\nabla h(X^{i}_{t})^{\top}\bigg( \delta\,\mathds{I}_{d} + D\big(X^{i}_{t} - \kappa\,\mathfrak{m}_{\alpha}(\rho^{N}_{t})\big) \bigg)\, \rd B^{i}_{t}\\
    & =: F(t,X_{t}^{i})\,\rd t + G(t,X_{t}^{i})\,\rd B^{i}_{t}.
\end{align*}

Let us introduce $\phi:[0,+\infty)\to \mathbb{R}$ such that $\phi'(t) = F(t, X^{i}_{t})$. In integral form, we have
\begin{equation*}
    \phi(t) - \phi(0)  = \int_{0}^{t}F(s,X^{i}_{s})\,\rd s = \frac{1}{p}\int_{0}^{t}\rd |X^{i}_{s}|^{p} - \int_{0}^{t} G(s,X^{i}_{s})\,\rd B^{i}_{s}.
\end{equation*}
One could observe that, thanks to Lemma \ref{thmexist} we have
\begin{equation*}
    \mathbb{E}\left[\int_{0}^{t} G(s,X^{i}_{s})\,\rd B^{i}_{s}\right] = 0,\quad \forall\, t>0.
\end{equation*}
Indeed, this is true since we have
\begin{align*}
    & \mathbb{E}\left[\int_{0}^{t} |G(s,X^{i}_{s})|^{2}\rd s\right] \leq C \, \EE\left[\int_0^t|X_s^i|^{2p}\rd s\right] \; \leq T\,C\,\EE\left[\sup_{t\in[0, T]}|X_t^i|^{2p}\right] \\
    & \quad \quad \quad \quad \leq C(T,\lambda,\sigma,\, \EE[|X_0^i|^{2p}])\; <+\infty,\, \text{ thanks to Lemma \ref{thmexist} and $\rho_0\in\mathscr{P}_{2p}(\RR^d)$.}
\end{align*}
In particular, the latter result holds for a fixed and finite time horizon $T<\infty$. 
Therefore we have 
\begin{equation*}
    \mathbb{E}[\phi(t) - \phi(0)] = \frac{1}{p}\mathbb{E}\left[\int_{0}^{t}\rd |X^{i}_{s}|^{p}\right] = \mathbb{E}\left[\frac{1}{p}|X^{i}_{t}|^{p}\right] - \mathbb{E}\left[\frac{1}{p}|X^{i}_{0}|^{p}\right],
\end{equation*}
hence $\frac{\rd}{\rd t}\mathbb{E}\left[\frac{1}{p}|X^{i}_{t}|^{p}\right] = \frac{\rd}{\rd t} \mathbb{E}[\phi(t)]$. 
Provided we could exchange the expectation and the derivation, we would have
\begin{equation}\label{differentiation}
    \frac{\rd}{\rd t}\mathbb{E}\left[\frac{1}{p}|X^{i}_{t}|^{p}\right] = \mathbb{E}[\phi'(t)] = \mathbb{E}[F(t, X^{i}_{t})], \quad \forall \, t>0.
\end{equation}
Let $t$ be fixed in $(0,+\infty)$. We have
\begin{equation*}
    \frac{\rd}{\rd t}\mathbb{E}[\phi(t)] = \lim\limits_{h\to 0} \frac{1}{h}\left(\mathbb{E}[\phi(t+h)] - \mathbb{E}[\phi(t)]\right) = \lim\limits_{h\to 0} \mathbb{E}[\phi'(\tau_{h})]
\end{equation*} 
where $\tau_{h}\in (t,t+h)$ is given by the mean-value theorem. We can now apply the dominated convergence theorem, noting that $|\phi'(\cdot)| = |F(\,\cdot\,,X^{i}_{\cdot})|$ is continuous, and is uniformly bounded in a (sufficiently large) neighborhood of $t$. This proves the claim \eqref{differentiation}. 
We can now further estimate the right-hand side as follows. Recall
\begin{align*}
    F(t,X^{i}_{t}) 
     &= -\lambda \nabla h(X^{i}_{t})\cdot \big(X^{i}_{t} - \kappa \, \mathfrak{m}_{\alpha}(\rho^{N}_{t})\big) \\
     &\quad+ \frac{\sigma^{2}}{2} \text{trace}\left[D^{2}h(X^{i}_{t}) \bigg( \delta\,\mathds{I}_{d} + D\big(X^{i}_{t} - \kappa\,\mathfrak{m}_{\alpha}(\rho^{N}_{t})\big) \bigg)^{2} \right]
\end{align*}
and with $h(x)=\frac{1}{p}|x|^{p}$ we have
\begin{align*}
    \nabla h(x) = |x|^{p-2}x, \quad \text{ and } \quad D^{2}h(x) = |x|^{p-2}\mathds{I}_{d} + (p-2)|x|^{p-4}x\otimes x.
\end{align*}
We deal first with the trace term in $F(t,X^{i}_{t})$. We have
\begin{align*}
    & 
    \text{trace}\left[D^{2}h(X^{i}_{t}) \bigg( \delta\,\mathds{I}_{d} + D\big(X^{i}_{t} - \kappa\,\mathfrak{m}_{\alpha}(\rho^{N}_{t})\big) \bigg)^{2} \right]\\
    =&\; \text{trace}\left[D^{2}h(X^{i}_{t}) \bigg( \delta^{2}\,\mathds{I}_{d} + D\big(X^{i}_{t} - \kappa\,\mathfrak{m}_{\alpha}(\rho^{N}_{t})\big)^{2} + 2\delta\, D\big(X^{i}_{t} - \kappa\,\mathfrak{m}_{\alpha}(\rho^{N}_{t})\big)\bigg) \right]\\
    =&\; \delta^{2}\,\text{trace}(D^{2}h(X^{i}_{t})) + \text{trace}\big[ D^{2}h(X^{i}_{t})D(X^{i}_{t} - \kappa\,\mathfrak{m}_{\alpha}(\rho^{N}_{t}))^{2}\big]\\
    &\;  + 2\delta\,\text{trace}\big[D^{2}h(X^{i}_{t})D(X^{i}_{t} - \kappa\,\mathfrak{m}_{\alpha}(\rho^{N}_{t}))\big]\\
    =&\;  (d+p-2)\delta^{2}\,|X^{i}_{t}|^{p-2} \\
    & \; +  |X^{i}_{t}|^{p-2} \sum\limits_{k=1}^{d} |(X^{i}_{t} - \kappa\,\mathfrak{m}_{\alpha}(\rho^{N}_{t}))_{k}|^{2} + (p-2) |X^{i}_{t}|^{p-4} \sum\limits_{k=1}^{d}|(X^{i}_{t})_{k}|^{2} |(X^{i}_{t} - \kappa\,\mathfrak{m}_{\alpha}(\rho^{N}_{t}) )_{k}|^{2}\\
    &\; + 2\delta\, |X^{i}_{t}|^{p-2} \sum\limits_{k=1}^{d}|(X^{i}_{t} - \kappa\, \mathfrak{m}_{\alpha}(\rho^{N}_{t}))_{k}| + 2\delta\,(p-2)|X^{i}_{t}|^{p-4} \sum\limits_{k=1}^{d} |(X^{i}_{t})_{k}|^{2} |(X^{i}_{t} - \kappa\,\mathfrak{m}_{\alpha}(\rho^{N}_{t}))_{k}|.
\end{align*}
We use $ |(X^{i}_{t})_{k}|^{2} \leq |X^{i}_{t}|^{2}$. Also, in the last line, we use $|(X^{i}_{t} - \kappa\, \mathfrak{m}_{\alpha}(\rho^{N}_{t}))_{k}| \leq \frac{1}{2}(|(X^{i}_{t} - \kappa\, \mathfrak{m}_{\alpha}(\rho^{N}_{t}))_{k}|^{2} + 1)$, then the sum yields $\sum\limits_{k=1}^{d}|(X^{i}_{t} - \kappa\, \mathfrak{m}_{\alpha}(\rho^{N}_{t}))_{k}| \leq \frac{1}{2}(|X^{i}_{t} - \kappa\, \mathfrak{m}_{\alpha}(\rho^{N}_{t})|^{2} + d)$.  We obtain
\begin{align*}
    & 
    \text{trace}\left[D^{2}h(X^{i}_{t}) \bigg( \delta\,\mathds{I}_{d} + D\big(X^{i}_{t} - \kappa\,\mathfrak{m}_{\alpha}(\rho^{N}_{t})\big) \bigg)^{2} \right]\\
    \leq& \left\{ (d+p-2)\delta^{2} + d(p-1)\delta + (p-1)\left(1+\delta\right)|X^{i}_{t} - \kappa\,\mathfrak{m}_{\alpha}(\rho^{N}_{t})|^{2}\right\} |X^{i}_{t}|^{p-2}\\
    \leq& \left\{ (d+p-2)\delta^{2} + d(p-1)\delta + 2(p-1)\left(1+\delta\right)(|X^{i}_{t}|^{2} + \kappa^{2}\,|\mathfrak{m}_{\alpha}(\rho^{N}_{t})|^{2})\right\} |X^{i}_{t}|^{p-2}\\
    \leq&\left\{ (d+p-2)\delta^{2} + d(p-1)\delta\right\}|X^{i}_{t}|^{p-2} 
    + 2(p-1)\left(1+\delta\right) |X^{i}_{t}|^{p} +
     \\
    & +2(p-1)\left(1+\delta\right)\kappa^{2}\,|\mathfrak{m}_{\alpha}(\rho^{N}_{t})|^{2}\, |X^{i}_{t}|^{p-2}
\end{align*}
Therefore we have
\begin{align*}
    F(t,X^{i}_{t}) 
    & \leq -\lambda\left(1 -\frac{\kappa}{2}\right) |X^{i}_{t}|^{p} + \lambda\,\frac{\kappa}{2}\,|X^{i}_{t}|^{p-2}\,|\mathfrak{m}_{\alpha}(\rho^{N}_{t})|^{2} \\
    & \quad \quad +\frac{\sigma^{2}}{2}\left\{ (d+p-2)\delta^{2} + d(p-1)\delta\right\}|X^{i}_{t}|^{p-2} 
    + \sigma^{2}(p-1)\left(1+\delta\right) |X^{i}_{t}|^{p} 
     \\
    & \quad \quad \quad \quad +\sigma^{2}(p-1) \left( 1+\delta \right) \kappa^{2} \, |\mathfrak{m}_{\alpha} (\rho^{N}_{t})|^{2}\, |X^{i}_{t}|^{p-2}\\
    & \leq \left\{ -\lambda\left(1 -\frac{\kappa}{2}\right) + \sigma^{2}(p-1)\left(1+\delta\right) \right\}\,|X^{i}_{t}|^{p}  \\
    & \quad \quad + \left\{ \lambda\,\frac{\kappa}{2} + \sigma^{2}(p-1)\left(1+\delta\right)\kappa^{2} \right\}\,|\mathfrak{m}_{\alpha}(\rho^{N}_{t})|^{2}\, |X^{i}_{t}|^{p-2}\\
    & \quad \quad \quad \quad + \frac{\sigma^{2}}{2}\left\{ (d+p-2)\delta^{2} + d(p-1)\delta\right\}|X^{i}_{t}|^{p-2}.
\end{align*}

We can now use \eqref{lemeq2} with $q=2$ therein, which ensures the existence of $C_{1}>0$ depending only on the function $f(\cdot)$ and on $\alpha$, such that we have
\begin{equation*}
    |\mathfrak{m}_{\alpha}(\rho^{N}_{t})|^{2} \leq \frac{C_{1}}{N}\sum\limits_{j=1}^{N}|X^{j}_{t}|^{2}.
\end{equation*}
Therefore we have
\begin{equation}\label{eq: intermediate}
    \begin{aligned}
        |X^{i}_{t}|^{p-2} \,|\mathfrak{m}_{\alpha}(\rho^{N}_{t})|^{2} 
        & =  \frac{C_{1}}{N} \, \left(|X^{i}_{t}|^{p}+ \, \sum\limits_{j=1,j\neq i}^{N}|X^{i}_{t}|^{p-2}|X^{j}_{t}|^{2} \right).
    \end{aligned}
\end{equation}
Using H\"older's inequality ($\frac{p-2}{p} + \frac{2}{p} =1$ and $p>2$), one gets
\begin{equation*}
    \begin{aligned}
        \mathbb{E}\left[|X^{i}_{t}|^{p-2}|X^{j}_{t}|^{2}\right] & \leq \bigg(\mathbb{E}\left[(|X^{i}_{t}|^{p-2})^{\frac{p}{p-2}}\right]\bigg)^{\frac{p-2}{p}}\, \bigg(\mathbb{E}\left[(|X^{j}_{t}|^{2})^{\frac{p}{2}}\right]\bigg)^{\frac{2}{p}}\\
        & = \bigg(\mathbb{E}\left[|X^{i}_{t}|^{p}\right]\bigg)^{\frac{p-2}{p}}\, \bigg(\mathbb{E}\left[|X^{j}_{t}|^{p}\right]\bigg)^{\frac{2}{p}}\\
        & = \mathbb{E}\left[|X^{i}_{t}|^{p}\right] \quad \text{ (using the symmetry of the particles)}
    \end{aligned}
\end{equation*}
Taking the expectation in \eqref{eq: intermediate}, and together with the latter observation, ones gets
\begin{equation*}
    \begin{aligned}
        \mathbb{E}\left[|X^{i}_{t}|^{p-2} \,|\mathfrak{m}_{\alpha}(\rho^{N}_{t})|^{2}\right] 
        & \leq \frac{C_{1}}{N} \, \left(\mathbb{E}\left[|X^{i}_{t}|^{p}\right] + \sum\limits_{j=1,j\neq i}^{N} \mathbb{E}\left[|X^{i}_{t}|^{p-2}|X^{j}_{t}|^{2}\right] \right)  \leq C_{1}\, \mathbb{E}\left[|X^{i}_{t}|^{p}\right]. 
    \end{aligned}
\end{equation*}
On the other hand, using  Jensen's inequality, one gets for all $ 3\leq p<\infty $
\begin{align*}
    \mathbb{E}[|X^{i}_{t}|^{p-2}] & = \mathbb{E}\left[(|X^{i}_{t}|^{p})^{\frac{p-2}{p}}\right] \\
    & \leq \left(  \mathbb{E}\left[ |X^{i}_{t}|^{p}\right] \right)^{\frac{p-2}{p}} 
    \leq 
    \left\{
        \begin{aligned}
            & 1, \quad \text{ if }\quad \mathbb{E}\left[ |X^{i}_{t}|^{p} \right] \leq 1\\
            & \mathbb{E}\left[ |X^{i}_{t}|^{p} \right], \quad \text{otherwise.}
        \end{aligned}
    \right. \quad
    \leq 1 + \mathbb{E}\left[ |X^{i}_{t}|^{p}\right].
\end{align*}
Therefore 
\begin{align*}
    \mathbb{E}\left[F(t,X^{i}_{t}) \right]
    & \leq \bigg\{ -\lambda\left(1 -\frac{\kappa}{2}\right) + \sigma^{2}(p-1) + C_{1}\,\lambda\,\frac{\kappa}{2} + C_{1}\,\sigma^{2}(p-1)\kappa^{2}\\
    & \quad \quad \quad \quad \quad \quad + \frac{\sigma^{2}}{2}\left\{ (d+p-2)\delta^{2} + d(p-1)\delta\right\} \bigg\}\,\mathbb{E}\left[|X^{i}_{t}|^{p}\right]\\
    & \quad \quad \quad \quad \quad \quad \quad \quad \quad \quad + \frac{\sigma^{2}}{2}\left\{ (d+p-2)\delta^{2} + d(p-1)\delta\right\}.
\end{align*}
This finally yields
\begin{align*}
    \frac{\rd}{\rd t}\mathbb{E}\left[\frac{1}{p}|X^{i}_{t}|^{p}\right] & = \mathbb{E}\left[F(t,X^{i}_{t})\right]
     \leq -\gamma\,\mathbb{E}\left[|X_{t}^{i}|^{p}\right] + C \quad \forall\, 3<p<\infty
\end{align*}
where
\begin{align*}
    & \gamma :=  \lambda\left(1 -\frac{\kappa}{2}\right) - \sigma^{2}(p-1) - C_{1}\,\lambda\,\frac{\kappa}{2} - C_{1}\,\sigma^{2}(p-1)\kappa^{2} - C\\
    & \text{and } \quad C = \frac{\sigma^{2}}{2}\left\{ (d+p-2)\delta^{2} + d(p-1)\delta\right\}.
\end{align*}
Now we must choose $\lambda,\sigma$, and $\kappa$ in order to guarantee $\gamma>0$. Let us first recall that $\kappa\in (0,1)$, hence $\lambda(1-\kappa/2)>\lambda/2$. Then, for $\gamma>0$ it suffices to have
\begin{align*}
    &  \frac{\lambda}{2} - \sigma^{2}(p-1) - C_{1}\,\lambda\,\frac{\kappa}{2} - C_{1}\,\sigma^{2}(p-1)\kappa^{2} - C >0\\
    \Leftrightarrow \quad & \frac{\lambda}{2}\left(1 -C_{1}\,\kappa\right) - \sigma^{2}(p-1)\left(1+C_{1}\,\kappa^{2}\right) - C >0.
\end{align*}
If we now choose $\kappa<\frac{1}{2\,C_{1}}$, then $1-C_{1}\,\kappa>1/2$ and $1+C_{1}\,\kappa^{2}< \frac{1}{4\,C_{1}}+ 1$. Thus we have
\begin{align*}
    & \frac{\lambda}{2}\left(1 -C_{1}\,\kappa\right) - \sigma^{2}(p-1)\left(1+C_{1}\,\kappa^{2}\right) - C  > \frac{\lambda}{4} - \sigma^{2}(p-1)\left(1+\frac{1}{4\,C_{1}}\right) - C.
\end{align*}
Sufficient conditions for the latter to be positive is to have
\begin{equation}
\label{configuration}
    \lambda > \frac{\sigma^{2}(p-1)}{C_{1}} + 4 \,C,\quad \quad \text{ and } 0<\kappa < \frac{1}{2\,C_{1}}.
\end{equation}
Ultimately, one gets for all $3\leq p<\infty$
\begin{equation*}
    \frac{\rd}{\rd t}\mathbb{E}\left[|X^{i}_{t}|^{p}\right] \leq -\gamma\,\mathbb{E}\left[|X_{t}^{i}|^{p}\right] + C
\end{equation*}
which together with Gr\"{o}nwall's inequality yields
\begin{equation*}
    \mathbb{E}\left[ |X^{i}_{t}|^{p} \right] \leq \mathbb{E}\left[ |X^{i}_{0}|^{p} \right]\, e^{-\gamma\,t } + C \, \leq \mathbb{E}\left[ |X^{i}_{0}|^{p} \right] + C.
\end{equation*}
The conclusion follows noting that the latter is true for any arbitrarily fixed $t>0$, and $\mathbb{E}\left[ |X^{i}_{0}|^{p} \right] = \mathbb{E}\left[ |X^{j}_{0}|^{p} \right]$ for all $i,j \in [N]$.

The case of $p=2$ is trivial from our first bound on $F(t,X^{i}_{t})$. 
\end{proof}

The following lemma is instrumental for our main result. In particular, it will ensure a uniform--in--time estimate of the consensus point.

\begin{lemma}{\cite[Theorem 1, \& \S3.1]{doukhan2009evaluation}}\label{lemconver}
Let $(\omega_j,V_j)_{j\in\NN}$ be i.i.d. random variables with values in $\RR\times\RR^d$ with $\omega_j>0$ almost surely. Define 
\begin{align*}
    & \hat{\mc{W}}_N=\frac{1}{N}\sum_{j=1}^{N}\omega_j, \quad \hat{\mc{V}}_N=\frac{1}{N}\sum_{j=1}^{N}\omega_jV_j,\quad \hat{\mc R}_N=\frac{\hat{\mc{V}}_N}{\hat{\mc{W}}_N},\,\\
    & \text{and } \quad \mc W=\EE[\hat{\mc{W}}_N], \quad \quad  \mc{V}=\EE[\hat{\mc{V}}_N], \quad \quad  \mc{R}=\frac{\mc V}{\mc W}.
\end{align*}
Let $q>2$ and assume that for some $c>0$ it holds
\begin{equation*}
    \EE[|\omega_1|^q]\leq c,\quad \EE[|V_1|^{r_1}]\leq c,\quad \EE[|\omega_1V_1|^{r_2}]\leq c,\quad r_1:=\frac{2(q+2)}{q-2},\quad r_2:=\frac{2q}{q-2}
\end{equation*}
Then there exists some $C>0$ depending only on $c,q,r_1,r_2$ such that
\begin{equation*}
    \EE[|\hat{\mc R}_N-\mc R|^2]\leq C\frac{1}{N}\,.
\end{equation*}
\end{lemma}

\begin{lemma}\label{lembarrhoN}
	Consider i.i.d.  particles  $\{(\OX_t^i)_{t\geq 0}\}_{i=1}^N$  with $\rho_0^{\otimes N}-$ distributed initial data \eqref{CBO kappa}, and let $\bar \rho_t^N$ be the corresponding empirical measure. Denote by $\rho_t$ the law of $\OX_t^i$ for any $i\in[N]$ satisfying $\sup_{t\geq 0}\rho_t(|\cdot|^4)<\infty$. Then it holds that
	\begin{equation*}
		\sup_{t\geq 0}\,\EE[\,|\mm_\alpha(\bar\rho_t^N)-\mm_\alpha(\rho_t)|^2\,]\,\leq\, C_3\frac{1}{N}\,,
	\end{equation*}
	where $C_3$ depends only on $\sup_{t\geq 0}\rho_t(|\cdot|^4),\underline f,\alpha$
\end{lemma}

\begin{proof}
	We apply Lemma \ref{lemconver} with $(\omega_j^t,V_j^t)=(e^{-\alpha f(\OX_t^j)},\OX_t^j)$,\, $q=6$,\, $r_1=4$,\, and $r_2=3$.  Since $|\omega_j^t|\leq e^{-\alpha\underline f}$, it obviously holds that $\sup_{t\geq 0}\EE[|\omega_1^t|^6]\leq e^{-6\alpha\underline f}$. Notice that $\EE[|V_1^t|^4]=\EE[|\OX_t^1|^4]\leq C_2 $, then one can also check
	\begin{equation*}
		\sup_{t\geq 0}\EE[\,|\omega_1^tV_1^t|^3\,] \; \leq \; \sup_{t\geq 0}\EE[\,|\omega_1^t|^{12}\,]^{\frac{1}{4}}\;\sup_{t\geq 0}\EE[\,|V_1^t|^4\,]^{\frac{3}{4}}<\infty\,.
	\end{equation*}
	Consequently the assumptions  of Lemma \ref{lemconver} are satisfied, and the statement is obtained.
\end{proof} 
\begin{remark}
    According to Theorem \ref{thm: finite moments}, the assumption $\sup_{t\geq 0}\rho_t(|\cdot|^4)<\infty$ can be guaranteed by choosing $\rho_0\in \mathscr{P}_8(\RR^d)$.
\end{remark}

This lemma leads to the following uniform--in--time mean--field limit estimate which is the main result of this manuscript. 
\begin{theorem}\label{thmmean}
	Assume that  $f(\cdot)$ satisfy Assumption \ref{assum1}. Let $\{(X_t^i)_{t\geq 0}\}_{i=1}^N$ and $\{(\OX_t^i)_{t\geq 0}\}_{i=1}^N$ be the solutions to the interacting particle system \eqref{CBOkappa particle} and the mean-field dynamics \eqref{CBO kappa} respectively  with the same initial data $\{(X_0^i)\}_{i=1}^N$ (i.i.d. according to $\rho_0\in \mathscr{P}_{16}(\RR^d)$) and the same Brownian motions $\{(B_t^i)_{t\geq 0}\}_{i=1}^N$. Then, for sufficiently small $\kappa$ and large $\lambda$, for example $\lambda > 3\,\sigma^{2}$ and $\kappa<\frac{1}{2(1+L)}$,  where $L$ depends only on $C_2,L_f,s$ and $\alpha$, there  exists some $C_{\mathrm{MFL}}>0$ such that it holds
	\begin{equation}
		\sup\limits_{t\geq 0}\,\sup\limits_{i=1,\dots,N}\,\EE[\,|X_t^i-\OX_t^i|^2\,]\,\leq\, C_{\mathrm{MFL}}N^{-1}\,.
	\end{equation}
	where  $C_{\mathrm{MFL}}$ depends only on $\lambda,\sigma,\kappa,C_2,C_3$ and $L$.
\end{theorem}

\begin{remark}
	In particular, notice that here $C_{\mathrm{MFL}}$ is independent of time $t$.
\end{remark}

To prove this theorem, we need the following lemma on the bound of the probability of large excursions.
\begin{lemma}{\cite[Lemma 2.5]{gerber2023mean} }\label{lemprob}
    Let $(Z_j)_{j\in\NN}$ be a family of i.i.d. $\RR$-valued random variables such that $\EE[|Z_1|^4]<\infty$. Then for all $S>\EE[|Z_1|]$, there exists a constant $C_4>0$ depending only on $S$ and $\EE[|Z_1|^4]$  such that for any $N\in\NN$ it holds
    \begin{equation*}
    \mathbb{P}\left(\frac{1}{N}\sum_{j=1}^{N}Z_j\geq S\right)\leq C_4\frac{1}{N^2}\,.
    \end{equation*}
\end{lemma}

\begin{proof}[Proof of Theorem \ref{thmmean}] 
We have 
\begin{align*}
    \rd (X^{i}_{t} - \overline{X}^{i}_{t}) & = - \lambda \bigg((X^{i}_{t} - \overline{X}^{i}_{t}) - \kappa \big(\mathfrak{m}_{\alpha}(\rho^{N}_{t}) - \mathfrak{m}_{\alpha}(\rho_{t})\big)\bigg)\,\rd t\\
    & \quad \quad \quad + \sigma\bigg( D(X^{i}_{t} - \kappa\, \mathfrak{m}_{\alpha}(\rho^{N}_{t})) - D(\overline{X}^{i}_{t} - \kappa\,\mathfrak{m}_{\alpha}(\rho_{t})) \bigg)\, \rd B^{i}_{t}.
\end{align*}
We introduce $h:\mathbb{R}^{d}\to \mathbb{R}$, such that $h(z) = \frac{1}{2}|z|^{2}$. Hence, $\nabla h(z) = z$ and $D^{2}h(z) = \mathds{I}_{d}$. 
Applying It\^o's formula leads to 
\begin{align*}
    \rd h(X^{i}_{t} - \overline{X}^{i}_{t}) = 
    & \bigg\{ -\lambda (X^{i}_{t} - \overline{X}^{i}_{t})\cdot \big[(X^{i}_{t} - \overline{X}^{i}_{t}) - \kappa(\mathfrak{m}_{\alpha}(\rho^{N}_{t}) - \mathfrak{m}_{\alpha}(\rho_{t}))\big]\\
    & \quad + \frac{\sigma^{2}}{2}\,\text{trace}\bigg[\bigg( D(X^{i}_{t} - \kappa\, \mathfrak{m}_{\alpha}(\rho^{N}_{t})) - D(\overline{X}^{i}_{t} - \kappa\,\mathfrak{m}_{\alpha}(\rho_{t})) \bigg)^{2}\bigg]\bigg\}\,\rd t\\
    & \quad \quad + \sigma (X^{i}_{t} - \overline{X}^{i}_{t})\cdot \bigg[D(X^{i}_{t} - \kappa\, \mathfrak{m}_{\alpha}(\rho^{N}_{t})) - D(\overline{X}^{i}_{t} - \kappa\,\mathfrak{m}_{\alpha}(\rho_{t}))\bigg]\,\rd B^{i}_{t}\\
    & =: \mathbf{F}_{t}\,\rd t + \mathbf{G}_{t}\, \rd B^{i}_{t}.
\end{align*}
The same arguments as in the beginning of the proof of Theorem \ref{thm: finite moments} (see in particular the proof of \eqref{differentiation} therein), allow us to write
\begin{align*}
    & \frac{\rd}{\rd t}\EE\left[|X^{i}_{t} - \overline{X}^{i}_{t}|^{2}\right]  = 2 \, \EE\left[\mathbf{F}_{t}\right]\\
    & \quad = -2\,\lambda\, \EE\left[|X^{i}_{t} - \overline{X}^{i}_{t}|^{2}\right] + 2\,\lambda\,\kappa\, \EE\left[(X^{i}_{t} - \overline{X}^{i}_{t})\cdot \big(\mathfrak{m}_{\alpha}(\rho^{N}_{t}) - \mathfrak{m}_{\alpha}(\rho_{t})\big)\right]\\
    & \quad \quad \quad \quad + \sigma^{2}\,\EE\left[\text{trace}\bigg(\big( D(X^{i}_{t} - \kappa\, \mathfrak{m}_{\alpha}(\rho^{N}_{t})) - D(\overline{X}^{i}_{t} - \kappa\,\mathfrak{m}_{\alpha}(\rho_{t})) \big)^{2}\bigg)\right]. 
\end{align*}
We can now estimate each of the terms above. We have
\begin{align*}
    & \EE\left[(X^{i}_{t} - \overline{X}^{i}_{t})\cdot \big(\mathfrak{m}_{\alpha}(\rho^{N}_{t}) - \mathfrak{m}_{\alpha}(\rho_{t})\big)\right] \leq \frac{1}{2}\EE\left[|X^{i}_{t} - \overline{X}^{i}_{t}|^{2}\right] + \frac{1}{2}\EE\left[|\mathfrak{m}_{\alpha}(\rho^{N}_{t}) - \mathfrak{m}_{\alpha}(\rho_{t})|^{2}\right]
\end{align*}
and
\begin{align*}
    & \EE\left[\text{trace}\bigg(\big( D(X^{i}_{t} - \kappa\, \mathfrak{m}_{\alpha}(\rho^{N}_{t})) - D(\overline{X}^{i}_{t} - \kappa\,\mathfrak{m}_{\alpha}(\rho_{t})) \big)^{2}\bigg)\right] \\
    & \quad =  \sum\limits_{k=1}^{d} \EE\left[
    \bigg( |\{X^{i}_{t} - \kappa\,\mathfrak{m}_{\alpha}(\rho^{N}_{t})\}_{k}| - |\{\overline{X}^{i}_{t} - \kappa\,\mathfrak{m}_{\alpha}(\rho_{t})\}_{k}| \bigg)^{2}
    \right]\\
    & \quad \leq \sum\limits_{k=1}^{d} \EE\left[
    \left|\big\{X^{i}_{t}-\overline{X}^{i}_{t}\big\}_{k} - \kappa\,\{\mathfrak{m}_{\alpha}(\rho^{N}_{t}) - \mathfrak{m}_{\alpha}(\rho_{t})\big\}_{k}\right| ^{2}
    \right]\\
    & \quad \leq 2 \sum\limits_{k=1}^{d} \EE\left[
    \left|\big\{X^{i}_{t}-\overline{X}^{i}_{t}\big\}_{k}\right|^{2} + \kappa^{2}\,\left|\{\mathfrak{m}_{\alpha}(\rho^{N}_{t}) - \mathfrak{m}_{\alpha}(\rho_{t})\big\}_{k}\right| ^{2}
    \right]\\
    & \quad \leq 2\, \EE\left[|X^{i}_{t} - \overline{X}^{i}_{t}|^{2}\right] + 2 \,\kappa^{2}\,\EE\left[|\mathfrak{m}_{\alpha}(\rho^{N}_{t}) - \mathfrak{m}_{\alpha}(\rho_{t})|^{2}\right] .
\end{align*}
Therefore we have
\begin{equation}\label{eq: estimate final proof}
\begin{aligned}
    \frac{\rd}{\rd t}\EE\left[|X^{i}_{t} - \overline{X}^{i}_{t}|^{2}\right]  & \leq (-2\,\lambda + \lambda\,\kappa\, +2\,\sigma^{2})\, \EE\left[|X^{i}_{t} - \overline{X}^{i}_{t}|^{2}\right] \\
    &\quad \quad \quad  + (\lambda\,\kappa\, + 2\,\sigma^{2}\,\kappa^{2})\EE\left[|\mathfrak{m}_{\alpha}(\rho^{N}_{t}) - \mathfrak{m}_{\alpha}(\rho_{t})|^{2}\right].
\end{aligned}
\end{equation}
We are left with the last term which we estimate as follows. We have
\begin{align*}
    \EE\left[|\mathfrak{m}_{\alpha}(\rho^{N}_{t}) - \mathfrak{m}_{\alpha}(\rho_{t})|^{2}\right] \leq 2\,  \EE\left[|\mathfrak{m}_{\alpha}(\rho^{N}_{t}) - \mathfrak{m}_{\alpha}(\overline{\rho}^{N}_{t})|^{2}\right] + 2\,\EE\left[|\mathfrak{m}_{\alpha}(\overline{\rho}^{N}_{t}) - \mathfrak{m}_{\alpha}(\rho_{t})|^{2}\right] 
\end{align*}
where $\overline{\rho}^{N}_{t}$ is the empirical measure corresponding to $\{\overline{X}^{i}_{t}\, : \, 1\leq i \leq N\}$. Using Lemma \ref{lembarrhoN} yields 
\begin{align*}
    \EE\left[|\mathfrak{m}_{\alpha}(\overline{\rho}^{N}_{t}) - \mathfrak{m}_{\alpha}(\rho_{t})|^{2}\right] \leq C_{3}\frac{1}{N}
\end{align*}
where $C_{3}$ is a positive constant depending on the data and independent of $t$ and of $N$. To estimate the term $\EE\left[|\mathfrak{m}_{\alpha}(\rho^{N}_{t}) - \mathfrak{m}_{\alpha}(\overline{\rho}^{N}_{t})|^{2}\right]$ and bypass its non-Lipschitz growth, we use Lemma \ref{lemprob}. Thus we introduce for each $t\in[0,\infty)$ the event 
\begin{equation}
    A_{N,t}:=\left\{\omega\in\Omega:~\frac{1}{N}\sum_{j=1}^N|\OX_t^j (\omega)|^2\geq S\right\}\,.
\end{equation}
In the present setting, using the notation ($Z_{j}$ and $S$) of Lemma \ref{lemprob},  we set  $Z_j^t=|\OX_t^j|^2$, and one can verify from the estimate \eqref{eq(2)} that 
\begin{align*}
    & \sup_{t\geq 0}\,\EE[\,|Z_1^t|^4\,]= \sup_{t\geq 0}\,\EE[\,|\OX_t^1|^8\,]\,<C_2 ,\\
    \text{and } \quad & \sup_{t\geq 0}\,\EE[\,|Z_1^t|\,]= \,\sup_{t\geq 0}\,\EE[\,|\OX_t^1|^2\,] \,<C_2=:S. 
\end{align*}
Note that here we require the initial condition $\rho_0\in\mathscr{P}_{16}(\RR^d)$.
Therefore, Lemma \ref{lemprob}  leads to the following bound
\begin{equation}
    \mathbb{P}(A_{N,t})\leq \frac{C_4}{N^2}\quad \forall~t\geq 0 \,,
\end{equation}
where $C_4>0$ depends only on $C_2$.
Then one splits the term
\begin{align*}
    & \EE[\,|\mm_\alpha(\rho_t^{N})-\mm_\alpha(\bar \rho_t^N)|^2\,] =\EE[\,|\mm_\alpha(\rho_t^{N})-\mm_\alpha(\bar \rho_t^N)|^2\,\textbf{I}_{\Omega /A_{N,t}}\,]+\EE[\,|\mm_\alpha(\rho_t^{N})-\mm_\alpha(\bar \rho_t^N)|^2\,\textbf{I}_{A_{N,t}}\,] \\
    &\quad \quad \leq \EE[\,|\mm_\alpha(\rho_t^{N})-\mm_\alpha(\bar \rho_t^N)|^2\,\textbf{I}_{\Omega /A_{N,t}}\,]\,+\,(\EE[\,|\mm_\alpha(\rho_t^{N})-\mm_\alpha(\bar \rho_t^N)|^4\,])^{1/2}(P(A_{N,t}))^{1/2} \\
    &\quad \quad \leq \EE[\,|\mm_\alpha(\rho_t^{N})-\mm_\alpha(\bar \rho_t^N)|^2\,\textbf{I}_{\Omega /A_{N,t}}\,]+4(C_2C_4)^{1/2}\frac{1}{N}\,,
\end{align*}
where in the last inequality we used
\begin{align*}
    \EE[\,|\mm_\alpha(\rho_t^{N})-\mm_\alpha(\bar \rho_t^N)|^4\,]\,\leq\, 8\,\EE[\,|\mm_\alpha(\rho_t^{N})|^4\,]\,+\,8\,\EE[\,|\mm_\alpha(\bar \rho_t^{N})|^4\,]\,<16\,C_2\,.
\end{align*}
Meanwhile, the improved stability \eqref{lemeq1} leads to
\begin{align*}
    \EE[\,|\mm_\alpha(\rho_t^{N})-\mm_\alpha(\bar \rho_t^N)|^2\,\textbf{I}_{\Omega /A_{N,t}}\,]
    & \leq L\,\EE[\,W_2(\rho_t^N,\bar\rho_t^N)^2\,]\\
    & \leq L\,\EE\left[\,\frac{1}{N}\sum_{i\in[N]}|X_t^i-\OX_t^i|^2\,\right]\,=\,L\,\EE[\,|X_t^i-\OX_t^i|^2\,]\,,
\end{align*}
where $L$ depends only on $C_2,L_f,\alpha$.
Thus we have
\begin{equation*}
\EE[\,|\mm_\alpha(\rho_t^{N})-\mm_\alpha(\bar \rho_t^N)|^2\,]\,\leq\,  L\,\EE[\,|X_t^i-\OX_t^i|^2\,]\,+\,4(C_2C_4)^{1/2}\frac{1}{N}\,.
\end{equation*}
Finally, using the latter estimates in \eqref{eq: estimate final proof} yields
\begin{equation*}
\begin{aligned}
    \frac{\rd}{\rd t}\EE\left[|X^{i}_{t} - \overline{X}^{i}_{t}|^{2}\right]  & \leq (-2\,\lambda + \lambda\,\kappa\, +2\,\sigma^{2})\, \EE\left[|X^{i}_{t} - \overline{X}^{i}_{t}|^{2}\right] \\
    &\quad \quad \quad  + (\lambda\,\kappa\, + 2\,\sigma^{2}\,\kappa^{2})\left(L\,\EE[|X_t^i-\OX_t^i|^2] + 4(C_2C_4)^{1/2}\frac{1}{N}\right)\\
    & \leq \bigg(-2\,\lambda +2\,\sigma^{2} + \kappa\,(\lambda + L\,\lambda\, + 2\,\sigma^{2}\,\kappa\,L) \bigg)\, \EE\left[|X^{i}_{t} - \overline{X}^{i}_{t}|^{2}\right]\\
    & \quad \quad \quad \quad \quad +\frac{4(\lambda\,\kappa\, + 2\,\sigma^{2}\,\kappa^{2})}{N}(C_2C_4)^{1/2}\\
    & =:\, -\vartheta\, \EE[|X_t^i-\OX_t^i|^2] + \frac{\Theta}{N}
\end{aligned}
\end{equation*}
where 
\begin{align*}
    \vartheta := 2\,\lambda -2\,\sigma^{2} - \kappa\,(\lambda + L\,\lambda\, + 2\,\sigma^{2}\,\kappa\,L), \quad \text{ and } \quad \Theta:= 4(\lambda\,\kappa\, + 2\,\sigma^{2}\,\kappa^{2})(C_2C_4)^{1/2} >0.
\end{align*}
Choosing $\kappa\in (0,1)$ small, and $\lambda>0$ large, guarantees $\vartheta>0$. For example, one could choose
\begin{equation}\label{condition lambda kappa}
    \lambda > 3\,\sigma^{2} \quad \text{ and } \quad  \kappa<\frac{1}{2(1+L)}.
\end{equation}
We can therefore apply Gr\"onwall's Lemma and obtain the desired estimate
\begin{equation*}
    \EE\left[|X^{i}_{t} - \overline{X}^{i}_{t}|^{2}\right] \leq  \EE\left[|X^{i}_{0} - \overline{X}^{i}_{0}|^{2}\right]\,e^{-\vartheta\, t} + \frac{\Theta}{N}.
\end{equation*}
The initial condition being the same, one gets the desired conclusion. 

Proof of \eqref{condition lambda kappa}: Observe first that $\vartheta>0$ is equivalent to  
\begin{equation}\label{equivalent condition}
\lambda + \lambda\big(1 - \kappa(1+L)\big) > 2\sigma^{2} + 2\sigma^{2}\kappa^{2}L.
\end{equation}
Hence, choosing $\kappa<\frac{1}{2(1+L)}$ is equivalent to having
\(
\lambda(1-\kappa(1+L)) > \frac{\lambda}{2}.
\)
Additionally,  choosing $\kappa<\frac{1}{2(1+L)}$ guarantees that 
\[
\kappa^{2}L< \frac{L}{4(1+L)^{2}} < \frac{2(1+L)}{4(1+L)^{2}} = \frac{1}{2(1+L)},
\]
which then means $2\sigma^{2}\kappa^{2}L < \frac{\sigma^{2}}{1+L}$. On the other hand, choosing $ \lambda > 3\sigma^{2}$ implies $\lambda > 2\sigma^{2} > \frac{2\sigma^{2}}{1+L}$ because $L>0$. Therefore we have $\frac{\sigma^{2}}{1+L} < \frac{\lambda}{2}$. In conclusion, we have 
\[
2\sigma^{2}\kappa^{2}L < \frac{\sigma^{2}}{1+L} < \frac{\lambda}{2} < \lambda(1-\kappa(1+L)).
\]
It suffices now to add $\lambda$ in the right hand-side, satisfying $\lambda> 3\sigma^{2}$, then obtain \eqref{equivalent condition}.
\end{proof}

We conclude our manuscript by mentioning how to modify the standard CBO algorithm in order to take into account the rescaled consensus point with the parameter $\kappa$. This is summarized in Algorithm \ref{alg: cbo k}. The rigorous proof of convergence in the large particle limit is obtained in \cite{huang2025faithful}.

\RestyleAlgo{ruled}
\begin{algorithm}
\caption{The (rescaled) CBO algorithm}\label{alg: cbo k}
\KwData{$N,\, \lambda, \, \sigma, \, \alpha, \, \kappa,\, \delta, \, f$ [parameters], $X^{i}_{0} \sim \rho_{0}$ [initial condition], $T$ [time horizon]} 

\hfill

\textbf{Initialization:} $t\gets 0$ and $X^{i}_{t}\gets X^{i}_{0}$; 

\hfill

\While{$t\leq T$}{
 Run standard CBO algorithm \Comment*[r]{Keeping $\kappa$ in the dynamics}
 $t\gets t+\Delta t$\;
}

\KwResult{$x^{*} \gets  \frac{\kappa^{-1}}{N}\sum\limits_{i=1}^{N}X^{i}_{T}$ \Comment*[r]{The expectation rescaled back with $\kappa^{-1}$}}
\end{algorithm}

\subsection*{Acknowledgments}
The authors wish to thank the two referees for their valuable comments which helped improve the paper.

\bibliographystyle{plain} 
\bibliography{bibliography}

\end{document}